\documentclass[11pt,reqno]{amsart}
\usepackage[foot]{amsaddr}

\usepackage{amsmath} 
\usepackage{amsfonts,amssymb,amsthm,bm}
\usepackage{amsrefs}

\usepackage{enumitem}
\usepackage[utf8]{inputenc}
\usepackage[T1]{fontenc}
\usepackage{stackrel}
\usepackage{stmaryrd}

\usepackage{fullpage}

\usepackage[dvipsnames]{xcolor}

\usepackage{hyperref}

\def\dd{\mathrm{\,d}} 
\newcommand\abs[1]{\left|#1\right|} 
\newcommand\normp[2]{\left[#1\right]_{#2}} 
\newcommand\normptilde[2]{\left\llbracket#1\right\rrbracket_{#2}} 

\DeclareMathOperator{\diam}{diam} 
\DeclareMathOperator{\dist}{\mathit{d}} 

\newcommand{\Chi}{\mathcal{X}} 
\let\emptyset\varnothing 
\DeclareMathOperator{\card}{\#} 
\DeclareMathOperator*{\essinf}{ess\,inf}

\DeclareMathOperator*{\loc}{loc} 
\DeclareMathOperator{\rad}{rad} 

\newcommand{\R}{\mathbb{R}}

\def\XXint#1#2#3{\mkern3mu{\setbox0=\hbox{$#1{#2#3}{\int}$ }
\vcenter{\hbox{$#2#3$ }}\kern-.6\wd0}}

\theoremstyle{plain}
\newtheorem{theorem}{Theorem}
\newtheorem{propo}[theorem]{Proposition}
\newtheorem{lemma}[theorem]{Lemma}

\numberwithin{theorem}{section}

\theoremstyle{remark}
\newtheorem{remark}[theorem]{Remark}

\theoremstyle{definition}
\newtheorem{dfn}[theorem]{Definition}

\begin{document}


\title{On the extension of Muckenhoupt weights in metric spaces}
\date{\today}
\author{Emma-Karoliina Kurki}
\address[E-K.K.]{Aalto University, Department of Mathematics and Systems Analysis, P.O. BOX 11100, FI-00076 Aalto, Finland}
\email{emma-karoliina.kurki@aalto.fi}

\author{Carlos Mudarra}
\address[C.M.]{Aalto University, Department of Mathematics and Systems Analysis, P.O. BOX 11100, FI-00076 Aalto, Finland}
\email{carlos.c.mudarra@jyu.fi, carlos.mudarra@aalto.fi}



\thanks{\emph{Acknowledgements.} E.-K. Kurki was funded by a young researcher's grant from the Emil Aaltonen Foundation. C. Mudarra acknowledges financial support from the Academy of Finland. We thank Juha Kinnunen for many helpful discussions.}
\begin{abstract}
A theorem by Wolff states that weights defined on a measurable subset of $\mathbb{R}^n$ and satisfying a Muckenhoupt-type condition can be extended into the whole space as Muckenhoupt weights of the same class. We give a complete and self-contained proof of this theorem generalized into metric measure spaces supporting a doubling measure. Related to the extension problem, we also show estimates for Muckenhoupt weights on Whitney chains in the metric setting. 
\end{abstract}

\maketitle

\section{Introduction}

On a metric space $X$ with a doubling measure, the familiar Muckenhoupt class $A_p$ consists precisely of those weights for which the Hardy-Littlewood maximal operator maps the weighted space $L^p(X, w\dd\mu)$ onto itself. 
In addition to $A_p$ weights being ubiquitous in harmonic analysis, weighted norm inequalities have applications in the study of regularity of certain partial differential equations. In order to deduce weighted Poincaré inequalities using the theory of global weights, we would like to extend Muckenhoupt weights defined on subsets to the entire space. 

Our main result is the following theorem that provides an abstract starting point for the investigation of extensions. It is the generalization to a metric-space context of a result due to Thomas H. Wolff. The result in $\mathbb{R}^n$ supposedly originates in an elusive preprint titled ``Restrictions of $A_p$ weights'', that to our knowledge remains unpublished. An outline of the Euclidean proof can be found in \cite{MR807149}, Theorem 5.6. However, the metric setting brings about technical challenges that are not present in the Euclidean case.
\begin{theorem}\label{maintheorem}
Let $X$ be a complete metric space with a doubling measure, $E\subset X$ a measurable set with $\mu(E)>0$, and $w$ a weight on $E$. Then, for $1<p<\infty$, the following statements are equivalent. 
\begin{enumerate}[label=\normalfont{(\roman*)}]
\item \label{condi} There exists a weight $W\in A_p(X)$ such that $W=w$ a.~e. on $E$; 
\item \label{condii} There exists an $\varepsilon>0$ such that 
\begin{equation}\label{intro:muck}
\sup_{\substack{B\subset X\\ B\text{ ball}}}\left(\frac{1}{\mu(B)}\int_{B\cap E}w^{1+\varepsilon}\dd\mu\right)\left(\frac{1}{\mu(B)}\int_{B\cap E}\left(\frac{1}{w^{1+\varepsilon}}\right)^\frac{1}{p-1}\dd\mu\right)^{p-1} <\infty.
\end{equation}
\end{enumerate}
In addition, whenever $p=1$, the condition \ref{condii} takes the following form: There exists a constant $C>0$ such that
$$\frac{1}{\mu(B)} \int_{B \cap E} w^{1+\varepsilon} \dd \mu \leq C \essinf_{B \cap E} w^{1+\varepsilon}$$
for every ball $B \subset X$. 
\end{theorem}
In Section \ref{section:wolff} below, we present a complete and self-contained proof of Wolff's extension theorem (above and Theorem \ref{thm:wolff}) for measurable sets in a metric space supporting a doubling measure. Comparing \ref{condii} to the classical $A_p$ condition, it is clear that we need to deal with weights and maximal functions restricted to arbitrary measurable subsets $E\subset X$. We have chosen to call these classes \emph{induced $A_p$ weights}; see Definition \ref{def:Aptilde} below. It is not obvious at the outset whether all properties of globally defined weights hold true for this class as well. 

Like the corresponding proof in $\mathbb{R}^n$, our proof relies on a factorization theorem, which in turn is based on the boundedness of the maximal operator. In particular, we need to show that the restricted maximal function is bounded on $L^p(E,w)$ when the weight $w$ belongs to the induced $A_q$ class for some $q<p$ (Theorem \ref{estimateAqeveryg}). The proofs of this theorem in the whole space are based either on Calderón-Zygmund decompositions on cubes (when $X=\mathbb{R}^n)$ or on Vitali-type coverings of the distributional sets of the maximal function. It is not clear how to adapt these arguments when $E$ is an arbitrary subset, because of the simple fact that the \textit{relative balls} $E\cap B$ do not necessarily satisfy a doubling condition, i.~e., $\mu( E \cap 5B )$ or $w(E \cap 5B )$ are not comparable to $\mu( E \cap B)$ or $w(E \cap B)$ in general.

The reader might wonder why we need to assume the Muckenhoupt-type condition \eqref{intro:muck} for $w^{1+\epsilon}$ instead of simply stating the corresponding condition for $w.$ A Muckenhoupt weight $W\in A_p(X)$ in the whole space always satisfies a self-improving property in the sense that $W^{1+\varepsilon}$ also belongs to $A_p(X)$ for a suitable $\varepsilon$ depending on the characteristic $A_p$ constant of $W$ (Lemma \ref{thm:simp}). This is a consequence of the fact that global Muckenhoupt weights satisfy a reverse Hölder inequality (RHI; Proposition \ref{thm:revho}). As a result, one is free to apply Gehring's lemma to obtain the desired self-improving inequality. However, it is unclear whether the induced classes of weights satisfy a RHI, since it is yet again impossible to control the measures of the relative balls $B\cap E$ in terms of those of $B$. Even when the measure is positive, $\mu(B \cap E)$ might be too small in comparison to $\mu(B)$, unless we tighten our assumptions on the set $E$. This technicality destroys our ability to compute the averaged integrals that would lead to the RHI. 

For an early treatment of harmonic analysis in metric spaces, see \cite{MR0499948}. Muckenhoupt weights in particular are discussed in \cite{MR1791462} and \cite{MR1011673}. A solid reference to the theory, albeit in $\mathbb{R}^n$, is \cite{MR807149}. For recent results concerning reverse Hölder inequalities for $A_\infty$ weights or strong $A_p$ weights, as well as versions of the Gehring lemma in various spaces, see the articles  \cites{MR3785798, AndersonHytonenTapiola2017, AuscherBortzEgertSaari2020, MR2078632, MR3544941, HytonenPerezRela2012, HytonenPerez2013, MR3265363, MR3130552, MR2815740, LuquePerezRela2017}.

In Sections \ref{section:balls} and \ref{section:weights}, we turn to an intended application of Theorem \ref{maintheorem}. This theorem gives a necessary and sufficient condition for the existence of an extension. One might ask what are the subsets $E$ and weights $w$ that satisfy \eqref{intro:muck} and consequently possess an extension to the entire space. Peter J. Holden \cite{MR1162041}, working in $\mathbb{R}^n$, has verified \eqref{intro:muck} for weights in $A_p(E)$  under additional geometric assumptions on the set $E$. We made an effort to reproduce Holden's argument in the metric setting, yet were not able to reach the point where we could apply Theorem \ref{maintheorem}. However, we believe that our findings are of independent interest and value for future research. In particular, Lemma \ref{lem:hold2} states that the weights of balls on a Whitney chain are comparable as long as we are able to control the length of the chain. In \cite{MR1162041}, this lemma is used to recover the Euclidean equivalent of Theorem \ref{maintheorem}~\ref{condii}. 

\section{The extension theorem}\label{section:wolff}

The results in this section apply in a complete metric measure space $\left(X,\dist,\mu\right)$. In addition, we assume that the nontrivial Borel regular measure $\mu$ satisfies the \emph{doubling condition}: there exists a constant $C_d=C_d(\mu) > 1$ only depending on $\mu$ such that 
\begin{equation}\label{Doubling}
0<\mu\left(2 B\right) \leq C_{d}  \mu\left( B \right) < \infty
\end{equation}
for all balls $B\subset X$. The constant $C_d$ is spoken of as the \emph{doubling constant}. In particular, we assume that every ball in $X$ has positive and finite measure. An argument involving the Vitali covering lemma shows that $X$ is separable and that every ball is totally bounded. This implies that $X$ is locally compact and proper, which in turn means that every closed and bounded subset of $X$ is compact.

A \emph{ball} is determined by its center $x$ and radius $r$ and denoted $B = B(x, r) = \lbrace y\in X \mathbin{:} \dist(x,y)<r \rbrace$, where the center and radius are left out when not relevant to the discussion. Observe that in general, the center and the radius of a ball $B$ are not uniquely determined by $B$ as a set. We use the notation $\rad(B) = r$ when $B=B(x,r)$ and, at times, $cB = B(x, cr)$ for the ball dilated by a constant $c$.

For any two nonnegative numbers $A$ and $B$, if there exists a constant $C\in \left(0,\infty\right)$ such that $A\leq CB$, we write $A\lesssim B$. Furthermore, we write $A\approx B$ whenever there exist constants $C_1, C_2\in \left(0,\infty\right)$ such that $C_1A\leq B \leq C_2A$. This notation is used where the exact magnitude of the constants is not of interest. 

Whenever $E\subset X$ is a measurable subset and the function $f$ is integrable on every compact subset of $E$ we say that $f$ is \emph{locally integrable} on $E$, denoted $f\in L^1_{\loc}(E)$. If the measure $\nu$ is absolutely continuous with respect to $\mu$ and if there exists a nonnegative locally integrable function $w$ such that $d\nu=w\dd\mu$, we call $\nu$ a \emph{weighted measure} with respect to $\mu$, and $w$ a \emph{weight}. \cite{MR1011673} In practice, we assume $w$ to be positive almost everywhere in $E$. For any measurable subsets $F\subset E$ and a weight $w$ on $E,$ we write $w(F)=\int_F w \dd \mu$.

For the purposes of Theorem \ref{thm:wolff}, we introduce the following classes of \emph{induced} Muckenhoupt weights on a subset, which we denote by $\widetilde{A}_p$. 
\begin{dfn}\label{def:Aptilde}On a metric space $X$, let $E\subset X$ be a measurable subset with $\mu(E)>0.$ Let $w$ be a weight on $E.$ If $1<p<\infty,$ we say that $w\in \widetilde{A}_p(E)$ whenever
\begin{equation}\label{Aptilde}
\normptilde{w}{p}=\sup_{\substack{B\subset X\\ B\text{ ball}}}\left(\frac{1}{\mu(B)}\int_{B\cap E}w \dd\mu\right)\left(\frac{1}{\mu(B)}\int_{B\cap E}\left(\frac{1}{w}\right)^\frac{1}{p-1}\dd\mu \right)^{p-1} <\infty.
\end{equation}
If $p=1,$ we define $\widetilde{A}_1(E)$ as the class of weights $w$ for which there exists $C>0$ with
\begin{equation}\label{A1tilde}
\frac{1}{\mu(B)} \int_{B \cap E} w  \dd \mu \leq C \essinf_{ B \cap E} w 
\end{equation}
for every ball $B\subset X$. We denote by $\normptilde{w}{1}$ the infimum of the $C>0$ for which the inequality \eqref{A1tilde} holds. 
\end{dfn}
Observe that conditions \eqref{Aptilde} and \eqref{A1tilde} imply that $w$ is integrable on each $B \cap E,$ where $B \subset X$ is a ball. Whenever $E=X$, the above classes coincide with Muckenhoupt weights as usually defined. In this case we will denote them by $A_p(X)$ and $A_1(X)$, respectively. Notice that it is not possible to reduce \eqref{Aptilde} to the $A_p(X)$ condition e.~g. by replacing $w$ with $\Chi_Ew$.

\begin{dfn}\label{def:Mf} The \emph{Hardy-Littlewood maximal function} is defined by
\begin{equation*}
Mf(x)=\sup_{B\ni x}\frac{1}{\mu(B)}\int_B\abs{f}\dd\mu,
\end{equation*}
where $X$ is a metric space, $B\subset X$ are balls and $f\in L^1_{\loc}(X)$. Whenever $E\subset X$ is a measurable set and $x\in E$, we also define a maximal function relative to the set $E$ by
\begin{equation*}
m_Ef(x) = \sup_{B\ni x}\frac{1}{\mu(B)}\int_{B\cap E}\abs{f}\dd \mu. 
\end{equation*}
\end{dfn}

In the following we verify a number of propositions regarding the $\widetilde{A}_p$ classes, leading to the proof of the extension theorem. These correspond to well-known results for $A_p$  weights in $\R^n$. While the proofs are based on those in $A_p(\R^n)$, we have chosen to present them in full, because our notion of induced weights along with the metric setting presents some difficulties that do not appear in $\mathbb{R}^n$.

Throughout the rest of this section $(X,d,\mu)$ will denote a complete metric measure space, with the measure $\mu$ satisfying doubling condition \eqref{Doubling} and thus all the properties mentioned at the beginning of the section.

To begin with, the following proposition is a generalization of a well-known result for $A_1(\mathbb{R}^n)$; see \cite{MR3243734}, p.~502. 
\begin{propo}\label{thm:maxim1}Let $E\subset X$ be a measurable set with $\mu(E)>0.$ If $w \in \widetilde{A}_1(E),$ then $w(x) \leq m_Ew(x)\leq \normptilde{w}{1} w(x)$ for almost every $x\in E$. 
\end{propo}
\begin{proof}
The first inequality is a consequence of the Lebesgue differentiation theorem for $\mu$. For a proof of this classical theorem in a metric space with a doubling measure see \cite{MR1800917}, p.~4. As for the second one, let $A= \lbrace x\in E \mathbin{:} m_E w(x) > \normptilde{w}{1} w(x) \rbrace.$ We aim to show that $\mu(A)=0.$ Because $X$ is separable, there exists a dense sequence of points $\lbrace z_k \rbrace_k$ in $X.$ Define a countable collection of balls $\mathcal{F} = \lbrace B(z_k, q) \mathbin{:} k \geq 1, q\in \mathbb{Q}^+ \rbrace.$ Then, for every $x\in A$, there exist a $\delta\in(0,1)$ and a ball $B = B(z,r)\ni x$ such that
$$
\normptilde{w}{1} w(x) < (1-\delta)\frac{1}{\mu(B)}\int_{B\cap E}w\dd\mu.
$$

For every $\varepsilon \in (0,1) $ denote $B_\varepsilon=B(z,(1-\varepsilon) r).$ Because $w$ is integrable on $B\cap E,$ by the absolute continuity of the Lebesgue integral there exists an $\eta>0$ such that if $F\subset B$ is measurable and $\mu(F) \leq \eta,$ then $\int_{F\cap E} w \leq \delta \int_{B\cap E} w.$ If $\varepsilon \in (0,1) $ is small enough so that $x \in B_\varepsilon$ and $\mu\left( B \setminus B_\varepsilon \right) \leq \eta$, then $\int_{{\left(B \setminus B_\varepsilon\right)\cap E}} w \leq \delta \int_{B\cap E} w.$ Let $B'=B(z', q)$, where $z' \in \lbrace z_k \rbrace_k$ and $q\in \mathbb{Q}^+ $ are chosen so that $\dist(z,z') < \varepsilon r/4$ and $(1-3 \varepsilon/4) r< q  < (1-\varepsilon/4) r.$ The triangle inequality gives the inclusions $B_ \varepsilon \subset B' \subset B,$ implying that $x\in B'$ and $\int_{\left(B \setminus B'\right)\cap E} w \leq \delta \int_{B\cap E} w .$ It follows that
$$
\int_{B\cap E} w\dd\mu =\int_{B'\cap E} w\dd\mu  + \int_{\left(B \setminus B'\right)\cap E} w\dd\mu  \leq \int_{B'\cap E} w\dd\mu  + \delta \int_{B\cap E} w\dd\mu
$$
which, recalling that $w\in \widetilde{A}_1(E)$, yields
$$
\normptilde{w}{1} w(x) < (1-\delta) \frac{1}{\mu(B)}\int_{B\cap E}w\dd\mu < \frac{1}{\mu(B')}\int_{B'\cap E}w\dd\mu \leq \normptilde{w}{1}\essinf_{B'\cap E}w.
$$
We have shown that $w(x) < \essinf_{B' \cap E} w$, which means that $x$ belongs to the set $D_{B'}= \lbrace y\in  B' \cap E \mathbin{:} w(y) < \essinf_{B' \cap E} w \rbrace,$ where $\mu(D_{B'})=0.$ Hence $A \subset \bigcup_{B' \in \mathcal{F}} D_{B'}$, which is a countable union of sets of measure zero. 
\end{proof}
We remark that the above proposition remains true if the maximal function $m_E$ is defined by taking a supremum over closed balls instead. The proof is similar, except that the absolute continuity of the Lebesgue integral is not needed.

The next lemma follows from Proposition \ref{thm:maxim1}, and is needed in the proof of Theorem \ref{thm:wolff} below.
\begin{lemma}\label{remarkfinitenessmaximal}
For a measurable set $E\subset X$ with $\mu(E)>0$ and a weight $w \in \widetilde{A}_1(E),$ the function $w \Chi_E$ is in $L^1_{\loc}(X)$ and its maximal function $M( w \Chi_E )$ is finite at almost every point of $X.$ Here $ w \Chi_E $ is the function in $X$ that coincides with $w$ on $E$ and vanishes outside $E.$
\end{lemma}
\begin{proof}
It is immediate that $w\Chi_E \in L^1_{\loc}(X)$ because $\int_{B \cap E} w$ is finite for every ball $B \subset X.$ As for the second statement, Proposition \ref{thm:maxim1} implies that $M( w \Chi_E )(x) = m_Ew(x) <\infty$ for a.~e. $x\in E.$ It remains to verify that $M( w \Chi_E )<\infty$ on $X\setminus E.$ Defining 
$$
A=\lbrace y\in E \mathbin{:} m_Ew(y)\leq \normptilde{w}{1} w(y)< \infty \rbrace,
$$ 
Proposition \ref{thm:maxim1} shows that $\mu(E \setminus A) =0,$ and therefore $M( w \Chi_E ) = M( w \Chi_A )$ on $X.$ For a.~e. $x\in X\setminus A,$ the Lebesgue differentiation theorem states that
$$
\lim_{\substack{B\ni x\\ r(B) \to 0}} \frac{1}{\mu(B)}\int_B w \Chi_A \dd \mu= w \Chi_A(x)=0,
$$
and thus there exists a radius $r_x>0$ such that 
$$
\sup_{\substack{B\ni x\\ r(B) \leq r_x }}\frac{1}{\mu(B)}\int_B w \Chi_A \dd \mu \leq 1.
$$
For almost every $x\notin A,$ we estimate the averages over balls $B$ such that $x\in B$, $r(B)>r_x$, and $B \cap A \neq \emptyset.$ For such a ball $B$ it clearly holds that $\dist(x,A)\leq 2 r(B),$ and consequently $r(B) \geq \max \lbrace r_x, \dist(x,A)/2 \rbrace.$ Also, there exists a point $y_0=y_0(x) \in A$ such that $\dist(x,y_0) < \max \lbrace r_x, 2d(x,A) \rbrace.$ Denoting by $z$ the center of $B$, we have 
$$
\dist(y_0,z) \leq \dist(y_0,x)+ \dist(x,z)<  \max \lbrace r_x, 2\dist(x,A) \rbrace + r(B) \leq 4r(B)+ r(B)=5r(B),
$$
and hence $y_0 \in 5B.$ Using the doubling condition for $\mu$ and the definition of $A,$ we obtain
$$
\frac{1}{\mu(B)}\int_B w \Chi_A \dd \mu \leq C(C_d) \frac{1}{\mu(5B)}\int_{5B} w \Chi_A \dd \mu \leq C(C_d) \sup_{B'\ni y_0}\frac{1}{\mu(B')}\int_{B' \cap A} w \dd \mu \leq C(C_d) \normptilde{w}{1} w(y_0).
$$
We conclude that $M( w \Chi_A )(x)< \infty$ for almost every $x\in X.$ 
\end{proof}

In the following two technical lemmas, we will not be using the fact that the measure is doubling.
\begin{lemma}\label{inequalitiesexponents}
Let $E\subset X$ be a measurable set with $\mu(E)>0$. If $p, q >1,$ $v \in\widetilde{A}_p(E),$ and $0\leq \delta \leq \min \lbrace 1, (q-1)(p-1)^{-1} \rbrace$, then $v^\delta \in \widetilde{A}_q(E)$ with $\normptilde{v^\delta}{q} \leq \normptilde{v}{p}^\delta $. Also, if $q\geq 1,$ $v \in\widetilde{A}_1(E),$ and $\delta\in [0,1],$ then $v^\delta \in \widetilde{A}_q(E)$ with $\normptilde{v^\delta}{q} \leq \normptilde{v}{1}^\delta $. In particular, $\widetilde{A}_p(E) \subset \widetilde{A}_q(E)$ for every $1 \leq p \leq q.$
\end{lemma}

\begin{proof}
We will use the following basic estimate. Let $A \subset X$ be measurable, $0 \leq s \leq 1$, and $h \in L^1(A)$. Then it follows from Hölder's inequality that
\begin{equation}\label{fromHolder}
\int_A h^s \dd\mu  \leq \mu(A)^{1-s} \left( \int_A h \dd\mu  \right)^s.
\end{equation}
Since the exponents $\delta$ and $\delta(p-1)(q-1)^{-1}$ are in $[0,1],$ we can apply \eqref{fromHolder} to obtain
\begin{align}
\label{tmpp}& \int_{B \cap E} v^\delta \dd\mu \leq \mu(B\cap E)^{1-\delta} \left( \int_{B \cap E} v \dd\mu \right)^\delta,\\
\label{tmppp} & \int_{B \cap E} \left( \frac{1}{v^\delta}\right)^\frac{1}{q-1} \dd\mu = \int_{B \cap E} \left( \frac{1}{v}\right)^\frac{\delta (p-1)}{(q-1)(p-1)} \dd\mu 
 \leq \mu(B\cap E)^{1-\frac{\delta (p-1)}{q-1}} \left( \int_{B \cap E} \left( \frac{1}{v}\right)^\frac{1}{p-1} \dd\mu \right)^{\frac{\delta (p-1)}{q-1}}.
\end{align}
Since the exponents $1-\delta$ and $(q-1)-\delta (p-1)$ are nonnegative, we have $\mu(B\cap E)^{1-\delta} \leq \mu(B)^{1-\delta}$ and $\mu(B\cap E)^{1-\delta (p-1)(q-1)^{-1}} \leq \mu(B)^{1-\delta (p-1)(q-1)^{-1}}.$ Then \eqref{tmpp} and \eqref{tmppp} lead to
\begin{align*}
& \frac{1}{\mu(B)^q}\int_{B \cap E} v^\delta \dd\mu  \left( \int_{B \cap E} \left( \frac{1}{v^\delta}\right)^\frac{1}{q-1} \dd\mu \right)^{q-1} \\
& \hspace*{2em} \leq \frac{\mu(B )^{q-\delta p} }{\mu(B)^q} \left( \int_{B \cap E} v \dd\mu \right)^\delta  \left( \int_{B \cap E} \left( \frac{1}{v}\right)^\frac{1}{p-1} \dd\mu \right)^{\delta (p-1)} \leq \normptilde{v}{p}^\delta,
\end{align*}
which proves the statement for $p>1$. In the case $\delta \in [0,1],$ $v\in \widetilde{A}_1(E),$ and $q > 1,$ using first \eqref{fromHolder} and then the definition of $\widetilde{A}_1(E)$ we can write
\begin{align*}
& \frac{1}{\mu(B)^q} \int_{B \cap E} v^\delta \dd\mu  \left( \int_{B \cap E} \left( \frac{1}{v^\delta}\right)^\frac{1}{q-1} \dd\mu \right)^{q-1} \leq \frac{\mu(B \cap E)^{1-\delta}}{\mu(B)^q}\left( \int_{B \cap E} v \dd\mu \right)^\delta \frac{\mu(B \cap E) ^{q-1}}{\essinf_{ B \cap E} v^\delta} \\
&\hspace*{2em} \leq \normptilde{v}{1}^\delta \left( \frac{\mu(B \cap E)}{\mu(B)} \right)^{q-\delta} \leq \normptilde{v}{1}^\delta,
\end{align*}
where we have used the fact that $ \delta \leq 1.$ For $q=1$, the result follows immediately from \eqref{fromHolder}.
\end{proof}

\begin{lemma}\label{estimateAqeveryg} Let $E\subset X$ be a measurable set with $\mu(E)>0.$ If $1\leq q<\infty,$ $v \in \widetilde{A}_q(E)$ and $g\in L^q(E, v),$ then for every ball $B\subset X$ we have 
$$
v(B \cap E) \left( \frac{1}{\mu(B)} \int_{B \cap E}\abs{g}  \dd\mu \right)^q \leq \normptilde{v}{q} \int_{B \cap E} \abs{g}^q v \dd \mu.
$$
\end{lemma}

\begin{proof}
We may and do assume $g \geq 0.$ In the case $q>1,$ applying Hölder's inequality we readily obtain
$$
 \left( \frac{1}{\mu(B)} \int_{B \cap E} g \dd\mu \right)^q \leq \frac{1}{\mu(B)^q} \int_{B \cap E} g^q v \dd\mu \left( \int_{B\cap E} \left( \frac{1}{v} \right)^{\frac{1}{q-1}} \dd\mu \right)^{q-1} \leq  \frac{\normptilde{v}{q}}{v(B\cap E)}\int_{B \cap E} g^q v \dd\mu .
$$
When $q=1,$ the assertion follows immediately from the definition of $\widetilde{A}_1(E)$ \eqref{A1tilde}.  
\end{proof}

The Hardy-Littlewood maximal function is well known to satisfy a weak type inequality. The following lemma provides a version for the maximal function relative to a subset. Notice that by letting $E=X$, we recover the classical result.
\begin{propo}\label{lem:weak}
Let $E\subset X$ be a measurable set with $\mu(E)>0.$ Furthermore, let $1\leq q <\infty$, $v\in \widetilde{A}_q(E)$, $f\in L^q(E, v)$, and $t>0.$ Then 
$$
v\left( \lbrace x\in E \mathbin{:} m_Ef(x) > t \rbrace \right) \leq C t^{-q} \int_E \abs{f}^q v \dd\mu,
$$
where the constant $C$ only depends on $q$, $\normptilde{v}{q}$, and the doubling constant $C_d(\mu)$.  
\end{propo}
\begin{proof}
We may assume $f \geq 0.$ We restrict the supremum defining the maximal function to balls with radius no greater than $R$, and denote the resulting function by $m_E^R f.$ Once we show the estimate for $m_E^R$, the statement will follow by the monotone convergence theorem.

For every $x\in E_t=\lbrace x\in E \mathbin{:} m_E^R f(x)>t \rbrace,$ there is a ball $B_x\ni x$ such that $\rad(B_x)\leq R$ and $\int_{B_x \cap E} f \dd\mu >t \mu(B_x).$ The set $E_t$ is contained in $\bigcup_{x\in E_t}B_x\cap E$. Since the space $X$ is separable, by the Vitali covering lemma we can find a disjoint sequence of balls $\lbrace B_j \rbrace_j$ belonging to this collection such that $\bigcup_{x\in E_t} B_x \subset \bigcup_{j} 5 B_j.$ Now let us write
\begin{equation}\label{baiyang}
\int_{E_t} v  \dd\mu \leq \int_{\bigcup_j (5 B_j \cap E)} v \dd\mu \leq \sum_j \int_{5B_j \cap E} v  \dd\mu .
\end{equation}
For each $j,$ we apply Lemma \ref{estimateAqeveryg} with $B=5 B_j$ and $g=f \Chi_{B_j \cap E}$ to deduce that the sum \eqref{baiyang} is smaller than 
\begin{align*}
& \normptilde{v}{q} \sum_j \int_{5B_j \cap E} g^q v \dd\mu \left( \frac{1}{\mu(5B_j)}\int_{5B_j \cap E} g \dd\mu \right)^{-q} =   \normptilde{v}{q} \sum_j \int_{B_j \cap E} f^q v \dd\mu \left( \frac{1}{\mu(5B_j)}\int_{B_j \cap E} f \dd\mu \right)^{-q}. 
\end{align*}
By the choice of the balls $B_x$, this in turn is smaller than
$$
\normptilde{v}{q} \sum_j \int_{B_j \cap E} f^q v \dd\mu\left( \frac{t \mu(B_j)}{\mu(5B_j)} \right)^{-q} 
\leq \normptilde{v}{q}C(q,C_d)  t^{-q} \int_{\bigcup_j B_j \cap E} f^q v \dd\mu \leq C t^{-q}\int_E f^q v \dd\mu,
$$
where we have applied the doubling property of $\mu$.   
\end{proof}

We next show a strong-type estimate for the maximal function $m_E$ restricted to $E,$ in the space $L^p(E, v),$ provided $v$ is an induced Muckenhoupt weight of a higher class.
\begin{propo}\label{boundedoperator} Let $E\subset X$ be a measurable set with $\mu(E)>0$, $1\leq q < p$, $v\in \widetilde{A}_q(E)$, and $f\in L^p(E, v).$ Then 
$$
\int_E (m_E f)^p v \dd\mu  \leq C \int_E \abs{f}^p v \dd\mu,
$$
where the constant $C$ depends only on $p, q$, $\normptilde{v}{q}$, and the doubling constant $C_d(\mu).$ 
\end{propo}
\begin{proof}
For simplicity, we again assume that $f \geq 0$, and proceed to write 
$$
f= f \Chi_{\lbrace f >t/2 \rbrace} + f \Chi_{\lbrace f \leq t/2 \rbrace} = f_t + f \Chi_{\lbrace f \leq t/2 \rbrace}.
$$
Using the subadditivity of the maximal function $m_E,$ we have that $m_E f \leq m_E(f_t)+ t/2,$ from which it is clear that the set $\lbrace x\in E \mathbin{:} m_Ef(x) > t \rbrace$ is contained in $ \lbrace x\in E \mathbin{:} m_Ef_t(x) > t/2 \rbrace .$ Combining this observation with Cavalieri's principle for the measure $v \, d\mu$, and then using Proposition \ref{lem:weak} for $q$ and $f_t$, we arrive at
\begin{align*}
& \int_E (m_E f)(x)^p v(x)\dd\mu(x) \leq p  \int_0 ^\infty t^{p-1} v\left( \lbrace x\in E \mathbin{:} m_Ef_t(x)>t/2 \rbrace \right) \dd t\\
& \hspace*{2em} \leq C \int_0^\infty t^{p-q-1} \int_E f_t(x)^q v(x) \dd\mu(x)\dd t = C \int_0 ^\infty t^{p-q-1} \int_{\lbrace x\in E \mathbin{:} f(x)>t/2 \rbrace} f(x)^q v(x) \dd\mu(x) \dd t \\
& \hspace*{2em} = C \int_E f(x)^q v(x) \int_0^{2f(x)} t^{p-q-1} \dd t\dd\mu(x) \leq C \int_E f(x)^p v(x) \dd\mu(x),
\end{align*}
where $C$ depends on $p, q,$ $\normptilde{v}{q}$, and $C_d(\mu).$ 
\end{proof}

The following factorization theorem will be one of the main ingredients in the proof of Theorem \ref{maintheorem}.

\begin{propo}\label{thm:joneses}
Let $E\subset X$ be a measurable set with $\mu(E)>0$, $p>1$, and $v$ a weight on $E$ such that $v^r \in \widetilde{A}_p(E)$ for some $r>1.$ Then there exist weights $v_1, v_2 \in \widetilde{A}_1(E)$ such that $v=v_1 v_2^{1-p}.$ 
\end{propo}
\begin{proof}
Writing $q_1= r^{-1}(p-1) +1,$ we have $1<q_1<p$ and, by virtue of Lemma \ref{inequalitiesexponents}, $v=(v^r)^{1/r} \in \widetilde{A}_{q_1}(E)$ with $\normptilde{v}{q_1} \leq \normptilde{v^r}{p}^{1/r}$. Also, by the hypothesis, the weight $(v^{-r})^{1/(p-1)}$ belongs to $\widetilde{A}_{p'}(E)$ with $\normptilde{(v^{-r})^{1/(p-1)}}{p'} = \normptilde{v^r}{p}^{1/(p-1)}$, where $p'$ is the conjugate exponent of $p$. Applying again Lemma \ref{inequalitiesexponents} for $q_2=r^{-1}(p'-1)+1$ and $\delta=r^{-1},$ we have that $v^{-1/(p-1)} \in \widetilde{A}_{q_2}(E)$ and $\normptilde{v^{-1/(p-1)}}{q_2} \leq \normptilde{(v^{-r})^{1/(p-1)}}{p'}^{1/r} \leq \normptilde{v^r}{p}^{1/r(p-1)}.$ Notice that $q_1<p$ and $q_2<p'$. 

Proposition \ref{boundedoperator} applied first with $v$ and $q_1$, and then with $v^{-1/(p-1)}$ and $q_2$, yields that $m_E$ is a bounded operator both in $L^p(E, v)$ and $L^{p'}\left( E, v^{-1/(p-1)} \right),$ with norms bounded by constants depending only on $r,$ $p$, and $\normptilde{v^r}{p}.$  

Let $v$ be as per the hypothesis and $p\geq 2$, and
$$
Tf = \left(v^{-\frac{1}{p}}m_E\left(v^\frac{1}{p}f^\frac{p}{p'}\right)\right)^\frac{p'}{p} + v^\frac{1}{p}m_E\left(v^{-\frac{1}{p}}f\right). 
$$
This is a bounded operator in $L^p(E)$, which can be verified by  applying Proposition \ref{boundedoperator}:
$$
\int_E \left(v^{-\frac{1}{p}}m_E\left(v^\frac{1}{p}f^\frac{p}{p'}\right)\right)^{\frac{p'}{p}\cdot p} \dd\mu = \int_E m_E\left(v^\frac{1}{p}f^\frac{p}{p'}\right)^{p'}v^{-\frac{1}{p-1}}\dd\mu \lesssim \int_E v^\frac{p'}{p}\abs{f}^pv^{-\frac{1}{p-1}}\dd\mu = \int_E \abs{f}^p\dd\mu,
$$
\begin{align*}
\int_E \left(v^\frac{1}{p}m_E\left(v^{-\frac{1}{p}}f\right)\right)^p\dd\mu = \int_E m_E\left(v^{-\frac{1}{p}}f\right)^p v\dd\mu(x) \lesssim \int_E \abs{f}^p\dd\mu. 
\end{align*}
Fix $f\in L^p(E)$ and set $\eta = \sum_{k=1}^\infty \left(2c\right)^{-k}T^kf$. The series converges absolutely, and by the completeness of $L^p(E)$, we conclude that $\eta\in L^p(E)$. The operator $T$ is subadditive since $p/p'\geq 1$, so 
$$
T\eta \leq \sum_{k=1}^\infty\left(2c\right)^{-k}T^{k+1}f = \sum_{k=2}^\infty\left(2c\right)^{1-k}T^{k}f \leq 2c\eta.
$$
It follows that the weights
$$
v_1 = v^\frac{1}{p}\eta^\frac{p}{p'},\quad v_2 = v^{-\frac{1}{p}}\eta
$$
are in $\widetilde{A}_1(E)$, because 
$$
m_Ev_1 \leq m_E(v^\frac{1}{p}\eta^\frac{p}{p'}) + v^\frac{1}{p}\left( v^\frac{1}{p}m_E\left(v^{-\frac{1}{p}}\eta\right) \right)^\frac{p}{p'} \leq v^\frac{1}{p}\left(T\eta\right)^\frac{p}{p'} \leq (2c)^\frac{p}{p'}v^\frac{1}{p}\eta^\frac{p}{p'} = (2c)^\frac{p}{p'}v_1,
$$
$$
m_Ev_2 = m_E\left(v^{-\frac{1}{p}}\eta\right) \leq v^{-\frac{1}{p}}\left(v^{-\frac{1}{p}}m_E\left(v^\frac{1}{p}\eta^\frac{p}{p'}\right)\right)^\frac{p'}{p} + m_E\left(v^{-\frac{1}{p}}\eta\right) = v^{-\frac{1}{p}}T\eta \leq v^{-\frac{1}{p}}2c\eta = 2cv_2.
$$

In the case $1<p<2$, we instead factorize $v^{1-p'} = v_1v_2^{1-p'}$ as above, and raise this equation to the power $1/(1-p')$. 
\end{proof}

The following proposition is one half of the Coifman-Rochberg characterisation of $A_1$ weights. We will not be needing the reverse statement. 
\begin{propo}\label{thm:maxime} Let $0<\varepsilon<1$, $g$ a nonnegative function such that $g, g^{-1}\in L^\infty(X)$, and $f\in L^1_{\loc}(X)$ a nonnegative function such that $Mf<\infty$ a.~e. in $X$. Then, the weight $g\left(Mf\right)^\varepsilon = w$ belongs to $A_1(X)$.
\end{propo}
\begin{proof}
Since $g, g^{-1}\in L^\infty(X)$, it is enough to show that for every ball $B\subset X$ and $x\in B$
\begin{equation}\label{tempa1}
\frac{1}{\mu(B)}\int_B (Mf)^\varepsilon\dd\mu \leq C(Mf)(x)^\varepsilon,
\end{equation}
where the constant $C$ depends on $\varepsilon$ and the doubling constant $C_d$. For a ball $B\subset X$, write $f = f\Chi_{4B}+f\left(1-\Chi_{4B}\right) = f_1 + f_2$. Then, owing to subadditivity of the maximal function, we have 
\begin{equation*}
\left(Mf\right)^\varepsilon \leq \left(Mf_1\right)^\varepsilon + \left(Mf_2\right)^\varepsilon.
\end{equation*}
Each term is estimated separately. Beginning with $\left(Mf_1\right)^\varepsilon$, by Cavalieri's principle we have
\begin{align}
\nonumber & \frac{1}{\mu(B)}\int_B (Mf_1)(y)^\varepsilon\dd\mu(y) = \frac{1}{\mu(B)}\int_0^\infty \varepsilon t^{\varepsilon-1}\mu\left(\left\{y\in B\mathbin{:} Mf_1(y)>t\right\}\right)\dd t\\
\label{kjh} & \hspace*{2em} = \frac{1}{\mu(B)}\left(\int_0^a\cdots + \int_a^\infty\cdots\right).
\end{align}
The first of these integrals can be estimated simply by 
\begin{equation*}
\frac{1}{\mu(B)}\int_0^a \varepsilon t^{\varepsilon-1}\mu\left(\left\{y\in B\mathbin{:} Mf_1(y)>t\right\}\right)\dd t \leq \frac{1}{\mu(B)}\int_0^a \varepsilon t^{\varepsilon-1}\mu\left(B\right)\dd t = a^\varepsilon.
\end{equation*}
As for the second, Proposition \ref{lem:weak} applied with $E=X$ for $f_1\in L^1(X)$ has it that 
\begin{align*}
& \frac{1}{\mu(B)}\int_a^\infty \varepsilon t^{\varepsilon-1}\mu\left(\left\{y\in B\mathbin{:} Mf_1(y)>t\right\}\right)\dd t 
\leq \frac{1}{\mu(B)}\int_a^\infty \varepsilon t^{\varepsilon-1}\mu\left(\left\{y\in X\mathbin{:} Mf_1(y)>t\right\}\right)\dd t \\
& \hspace*{2em} \leq \frac{1}{\mu(B)}\int_a^\infty \varepsilon t^{\varepsilon-1}\cdot \frac{C(\mu)}{t}\int_X\abs{f_1(y)}\dd\mu(y)\dd t = \frac{C(\mu)\varepsilon}{1-\varepsilon}a^{\varepsilon-1}\frac{1}{\mu(B)}\int_{4B}\abs{f(y)}\dd\mu(y).
\end{align*}
We choose $a = \mu(B)^{-1}\int_{4B}\abs{f}\dd\mu$ and combine the two parts. We may and do assume that $a$ is positive, as otherwise $f=0$ on $B$ and the desired inequality follows immediately. Then \eqref{kjh} becomes
\begin{align*}
& \frac{1}{\mu(B)}\int_B (Mf_1)(y)^\varepsilon\dd\mu(y) \leq \left(\frac{1}{\mu(B)}\int_{4B}\abs{f(y)}\dd\mu(y)\right)^\varepsilon\left(1+\frac{C(\mu)\varepsilon}{1-\varepsilon}\right)\\ 
& \hspace*{2em} =\left(1+\frac{C(\mu)\varepsilon}{1-\varepsilon}\right) \left(\frac{\mu(4B)}{\mu(B)}\frac{1}{\mu(4B)}\int_{4B}\abs{f(y)}\dd\mu(y)\right)^\varepsilon\\
& \hspace*{2em} \leq C(\mu, \varepsilon)\left(\frac{1}{\mu(4B)}\int_{4B}\abs{f(y)}\dd\mu(y)\right)^\varepsilon \leq C(Mf)(x)^\varepsilon,
\end{align*}
where we have used the fact that $\mu$ satisfies the doubling condition \eqref{Doubling}.

On to $\left(Mf_2\right)^\varepsilon$. Let $x,y\in B=B(z,r)$ and let $B'=B(z',r')$ be another ball containing $y.$ Assume first that there exists a point $p\in B' \setminus 4B.$ We claim that $r \leq r'.$ Indeed, otherwise we have $d(y,p) \leq 2r'\leq 2r$ and 
$$
d(y,z) \geq d(z,p)-d(p,y) \geq 4 r- 2r=2r> r,
$$
implying that $y\notin B,$ a contradiction. Using that $r \leq r',$ we have for any $q \in B$
$$
d(q,z') \leq d(q,y) + d(y,z') \leq 2r'+r'=3r',
$$
which shows that $B \subset 4 B'.$ In particular $x\in 4B'$ and we can write
\begin{equation*}
\frac{1}{\mu(B')}\int_{B'}\abs{f_2}\dd\mu \leq \frac{C(\mu)}{\mu(4B')}\int_{4B'}\abs{f_2}\dd\mu \leq C\sup_{B\ni x}  \frac{1}{\mu(B)}\int_{B}\abs{f}\dd\mu = C (Mf)(x). 
\end{equation*}
In the case $B' \subset 4B,$ we have that $\int_{B'}\abs{f_2}\dd\mu=0,$ and the preceding estimate trivially holds. In both cases, the right-hand side does not depend on the choice of $y$, and we have
\begin{equation*}
Mf_2(y) = \sup_{B\ni y}\frac{1}{\mu(B)}\int_{B}\abs{f_2}\dd\mu \leq C (Mf)(x),
\end{equation*}
which completes the proof of the proposition. 
\end{proof}

In order to show that \ref{condi} implies \ref{condii} in Theorem \ref{maintheorem}, we are going to need the self-improving property of classical $A_p(X)$ weights. This is Lemma \ref{thm:simp}, which is straightforward to prove with the following reverse Hölder inequality at hand. 
\begin{propo}\label{thm:revho} Let $1\leq p<\infty$, and $w\in A_p(X)$. Then there exist constants $\delta>0$ and $0<C<\infty$ such that for all balls $B\subset X$ we have
\begin{equation}\label{revho}
\left(\frac{1}{\mu(B)}\int_B w^{1+\delta}\dd\mu \right)^\frac{1}{1+\delta} \leq C\frac{1}{\mu(B)}\int_Bw \dd\mu . 
\end{equation}
\end{propo}
For a proof of Proposition \ref{thm:revho} see \cite{MR1011673}, Theorem I.15.
\begin{lemma}\label{thm:simp} Let $w\in A_p(X)$ with $1\leq p<\infty$. There exists an $\varepsilon>0$ such that $w^{1+\varepsilon}\in A_p(X)$.
\end{lemma}
\begin{proof}Let $\varepsilon>0$ be such that $w$ satisfies the reverse Hölder inequality \eqref{revho} with $\delta=\varepsilon$. If $p=1$, applying the said inequality \eqref{revho} and the $A_1$ condition \eqref{A1} of $w$ we have for any ball $B\subset X$
\begin{equation*}
\frac{1}{\mu(B)}\int_B w^{1+\varepsilon}\dd\mu \leq C\left(\frac{1}{\mu(B)}\int_Bw \dd\mu \right)^{1+\varepsilon}\leq C\bigl(\essinf_{B }w \bigr)^{1+\varepsilon} \leq C\essinf_{B }w^{1+\varepsilon},
\end{equation*}
which implies that $w^{1+\varepsilon}\in A_1(X)$. 

As for $p>1$ we start by observing that, as a consequence of Jensen's inequality, if a weight $v$ satisfies \eqref{revho} for some $\delta>0,$ then $v$ safisfies the same inequality for every $0<\delta'\leq \delta.$ It immediately follows from the $A_p$ condition \eqref{Aptilde} with $E=X$ that $w^{1-p'}\in A_{p'}(X)$ with $\frac{1}{p}+\frac{1}{p'}=1.$ As a consequence, we obtain that both $w$ and $w^{1-p'}$ satisfy a reverse Hölder inequality \eqref{revho} for $\varepsilon>0$ small enough. Together with the fact that $w\in A_p(X)$, this implies
\begin{align*}
& \frac{1}{\mu(B)}\int_B w^{1+\varepsilon}\dd\mu\left(\frac{1}{\mu(B)}\int_B w^{-\frac{1+\varepsilon}{p-1}}\dd\mu\right)^{p-1}\\
& \hspace*{2em} \leq C\left(\frac{1}{\mu(B)}\int_Bw\dd\mu\right)^{1+\varepsilon}\left(\frac{1}{\mu(B)}\int_Bw^{-\frac{1}{p-1}}\dd\mu\right)^{(1+\varepsilon)(p-1)}  \leq C\normptilde{w}{p}^{1+\varepsilon}, 
\end{align*}
which is the $A_p(X)$ condition for $w^{1+\varepsilon}$.
\end{proof}

We are now ready to prove our main result, Theorem \ref{maintheorem}.
\begin{theorem}\label{thm:wolff} Let $X$ be a complete metric space with a doubling measure, $E\subset X$ a measurable set with $\mu(E)>0$, and $w$ a weight on $E$. Then, for $1\leq p<\infty$, the following statements are equivalent. 
\begin{enumerate}[label=\normalfont{(\roman*)}]
\item \label{1.1} There exists a weight $W\in A_p(X)$ such that $W=w$ a.~e. on $E$; 
\item \label{1.2} There exists an $\varepsilon>0$ such that $w^{1+\varepsilon} \in \widetilde{A}_p(E).$ 
\end{enumerate}
\end{theorem}
\begin{proof}
The implication \ref{1.1} $\Rightarrow$ \ref{1.2} follows from Lemma \ref{thm:simp}. Because $W\in A_p(X)$ for a given $1\leq p<\infty$, there exists an $\varepsilon>0$ such that $W^{1+\varepsilon}\in A_p(X)$. Assume first that $p>1$. Then, for all balls $B\subset X$,
\begin{align*}
& \left(\frac{1}{\mu(B)}\int_{B\cap E}w^{1+\varepsilon}\dd\mu\right)\left(\frac{1}{\mu(B)}\int_{B\cap E}w^\frac{1+\varepsilon}{1-p}\dd\mu\right)^{p-1} \\
& \hspace*{2em} = \left(\frac{1}{\mu(B)}\int_{B\cap E}W^{1+\varepsilon}\dd\mu\right)\left(\frac{1}{\mu(B)}\int_{B\cap E}W^\frac{1+\varepsilon}{1-p}\dd\mu\right)^{p-1}\\
& \hspace*{2em} \leq \left(\frac{1}{\mu(B)}\int_{B}W^{1+\varepsilon}\dd\mu\right)\left(\frac{1}{\mu(B)}\int_{B}W^\frac{1+\varepsilon}{1-p}\dd\mu\right)^{p-1} \leq C. 
\end{align*}
If $p=1,$ it is enough to write
$$
\frac{1}{\mu(B)} \int_{B \cap E} w^{1+ \varepsilon} \dd \mu \leq \frac{1}{\mu(B)} \int_{B} W^{1+ \varepsilon} \dd \mu \leq C \essinf_{B } W^{1+\varepsilon} \leq C \essinf_{B \cap E} W^{1+\varepsilon}=C \essinf_{B \cap E} w^{1+\varepsilon}.
$$

Next, let us prove \ref{1.2} $\Rightarrow$ \ref{1.1}. Let us define the weight $v= w^{1+\frac{\varepsilon}{2}}$ on $E.$ Consider first the case $p>1.$ Because $w^{1+\varepsilon}\in \widetilde{A}_p(E),$ it is clear that $v$ satisfies the hypothesis of Proposition \ref{thm:joneses}, so we can write $v=v_1v_2^{1-p}$ on $E,$ where $v_1, v_2 \in \widetilde{A}_1(E).$ Next, we define 
\begin{equation*}
V_i = M\left(\Chi_Ev_i\right)^\delta, \quad i \in\left\{1,2\right\},\;\delta=\frac{1}{1+\frac{\varepsilon}{2}},
\end{equation*}
where $M$ is the Hardy--Littlewood maximal function, and $\Chi_Ev_i$ is the function in $X$ that coincides with $v_i$ on $E$ and vanishes outside $E.$ These are weights in $A_1(X)$ as per Lemma \ref{remarkfinitenessmaximal} and Proposition \ref{thm:maxime}. Then, $V_1V_2^{1-p}$ is again an $A_p(X)$ weight such that
\begin{equation*}
V_1V_2^{1-p} = \left(m_Ev_1\left(m_Ev_2\right)^{1-p}\right)^\delta 
\end{equation*}
on $E$, with the maximal function $m_E$ restricted to $E$ as per Definition \ref{def:Mf}. The fact that $v_1, v_2\in \widetilde{A}_1(E)$ implies that there is a constant $C=\max \lbrace \normptilde{v_1}{1}, \normptilde{v_2}{1} \rbrace $ such that $v_i\leq m_Ev_i \leq Cv_i$, $i=1,2$, almost everywhere on $E$ (Proposition \ref{thm:maxim1}). Thus there exist nonnegative functions $g_i, \, i=1,2,$ such that $g_i, g_i^{-1} \in L^\infty(X)$ and $g_i m_E v_i = v_i$ almost every where on $E.$ Defining $g= g_1 ^\delta g_2^{\delta(p-1)}$ we see that $g, g^{-1} \in L^\infty(X)$, $g>0$, and
\begin{equation*}
g(x)V_1(x)V_2(x)^{1-p} = \left(v_1(x)v_2(x)^{1-p}\right)^\delta = v(x)^\delta = w(x) 
\end{equation*}
for almost every $x\in E$. The weight $W = gV_1V_2^{1-p}$ is in $A_p(X)$ and satisfies $W=w$ a.~e. on $E$. 

Finally, if $p=1,$ we reproduce the above argument taking $v_1$ as $v$ and discarding the weight $v_2$.
\end{proof}

\section{Balls and chains}\label{section:balls}
The aim of this section is to collect several preparatory results concerning balls in a metric space with a doubling measure. Our reason to delve into the geometry of Whitney-type balls is that they can be used to give estimates for Muckenhoupt weights over chains. In particular, Lemma \ref{lem:whitneybolic} is needed to prove Lemma \ref{lem:hold2} in the next section, which in turn is an integral part of Holden's argument in \cite{MR1162041}. We have found it necessary to provide an explicit proof of Lemma \ref{lem:whitneybolic}, as we could not locate one in the literature.

While most results in this section do not require any additional assumptions, on occasion we need to assume the existence of geodesics joining every pair of points. To cite an example of geodesic spaces relevant to partial differential equations, Corollary 8.3.16 in \cite{MR3363168} states that a complete, doubling metric space that supports a Poincaré inequality admits a geodesic metric that is bilipschitz equivalent to the underlying metric, with 
constant depending on the doubling constant of the measure and the data of the Poincaré inequality.

We say that a complete metric space $(X,\dist)$ is a \emph{geodesic space} provided that any two points $x,y\in X$ can be joined by a continuous, rectifiable curve $\gamma: [a,b] \to X$ with $\dist(x,y)= \ell(\gamma)$, where $\ell(\gamma)$ denotes the length of $\gamma.$ A rectifiable curve $\gamma: [a,b] \to X$ satisfying $\ell(\gamma)=d(\gamma(a), \gamma(b))$ is called a \emph{geodesic} on $X.$ Note that for a general rectifiable curve $\gamma: [a,b] \to X,$ we always have the inequality $\ell(\gamma) \geq d(\gamma(a), \gamma(b)).$ 

We will invoke the following well-known property of geodesics: if $[a',b'] \subset [a,b],$ the subarc $\gamma_{|_{[a',b']}}$ of the geodesic $\gamma: [a,b] \to X$ is a geodesic too. Hence, for any three points $\gamma(t_i)$ on the geodesic $\gamma$ such that $a\leq t_0< t_1<t_2\leq b$, the triangle inequality for $\dist$ becomes an equality:
$$
\dist(\gamma(t_0), \gamma(t_2))= \dist(\gamma(t_0), \gamma(t_1))+ \dist(\gamma(t_1), \gamma(t_2)).
$$
Slightly abusing notation, we write $\gamma_{|_{[x_1,x_2]}}$ to mean $\gamma_{|_{[t_1,t_2]}}$ whenever $\gamma(t_i)=x_i,$ $i=1,2.$ 

Throughout the rest of this section, we will assume that $(X,d, \mu)$ is a complete metric measure space such that $\mu$ satisfies the doubling condition \eqref{Doubling}. Also, when using the notation $A \approx B$ or $A \lesssim B$ for any two real numbers $A, B,$ we understand that the constants involved may depend on the doubling constant $C_d(\mu).$

We begin by showing two lemmas in metric geometry for future reference. In the first one, the measure does not play any role.
\begin{lemma}\label{lem:ballintersection}
Let $X$ be a geodesic space, and $B$, $B'$ any two balls in $X$. Assume that $\rad(B) \lesssim \rad(B')$ and that $B'$ contains the center of $B$. Then there exists a ball $B'' \subset B \cap B'$ with $\rad(B'') \approx \rad(B)$. 
\end{lemma}
\begin{proof} By assumption, there is a constant $ 0<a \leq 1$ such that $a\rad(B) \leq \rad(B')$. In the first place, assume that $\dist(z,z') \leq \frac{1}{2} \rad(B)$. Let $z$ and $z'$ denote the centers of $B$ and $B'$ respectively. In this case, define $B''$ as the ball centered at $z'$ and of radius $\frac{a}{4}\rad(B)$. Since $a \rad(B) \leq \rad(B')$, it is obvious that $B'' \subset B'$. On the other hand, for any $x\in B''$ we can write
$$
\dist(x,z) \leq \dist(x,z') + \dist(z',z) \leq \tfrac{a}{4}\rad(B) + \tfrac{1}{2}\rad(B) < \tfrac{1}{4}\rad(B) + \tfrac{1}{2}\rad(B) < \rad(B),
$$
which shows that $B'' \subset B$, and we also have $\rad(B'') = \frac{a}{4}\rad(B) \approx \rad(B).$ 

Consider then the case $\dist(z,z') > \frac{1}{2}\rad(B).$ Let $\gamma$ be a continuous curve joining $z$ and $z'$ with $\ell(\gamma)=\dist(z,z')$. Because $ \frac{1}{2}\rad(B)<\dist(z,z') \leq \rad(B),$ there exists a point $p\in \gamma$ such that $\dist(p,z) = \frac{1}{2}\rad(B).$ Let $q \in \gamma$ be the midpoint between $z$ and $p$, that is, $\dist(z,q)=\dist(q,p)=\frac{1}{2}\dist(z,p)$. We define $B''$ as the ball centered at $q$ and radius $ \frac{1}{2}\dist(z,q) $. For any $x\in B''$ we have
$$
\dist(x,z) \leq \dist(x,q)+\dist(q,z) \leq \tfrac{1}{2}\dist(z,q)+\dist(q,z) < 2 \dist(z,q) =d(z,p)= \tfrac{1}{2}\rad(B).
$$
This shows that $B''\subset B.$ To verify that $B'' \subset B'$, notice first that $\dist(z,q)+\dist(q,z')=\dist(z,z')$ as subarcs of the geodesic $\gamma$. Now, for any $x\in B'',$ write
$$
\dist(x,z') \leq \dist(x,q)+\dist(q,z') \leq \tfrac{1}{2}\dist(z,q) + \dist(q,z') < \dist(z,z') \leq \rad(B'),
$$
whereby we conclude that $B''\subset B \cap B'.$ Finally, because $\dist(z,p)=\tfrac{1}{2}\rad(B),$ we have 
$$
\rad(B'') = \tfrac{1}{2}\dist(z,q) = \tfrac{1}{4}\dist(z,p) = \tfrac{1}{8}\rad(B),
$$
which completes the proof of the lemma. 
\end{proof}

\begin{lemma}\label{lem:rmu}
Let $B,B'\subset X$ any two balls such that $\rad(B)\approx\rad(B')$ and $\dist(p,p')\lesssim\rad(B)$ for some $p\in B,$ $p'\in B'.$ Then $\mu(B)\approx\mu(B')$. 
\end{lemma}
\begin{proof}
Let $z_B$ denote the center of $B.$ For any $x\in B'$ we have
\[\dist(x, z_B) \leq \dist(x, p') + \dist(p', p) + \dist(p, z_B) \leq 2 \rad(B') + 2\dist(B,B') + \rad(B) \lesssim \rad(B).
\]
As a result, there exists a constant $1\leq  \lambda <\infty$ such that $\dist(x, z_B) \leq \lambda \rad(B)$ for every $q\in B'$, which means that $B'\subset \lambda B$ and therefore $\mu(B')\leq\mu(\lambda B)$. But $\mu(\lambda B)\lesssim \mu(B)$ because the measure is doubling, so $\mu(B')\lesssim \mu(B)$. Reversing the roles of $B$ and $B'$ gives the inequality in the other direction. 
\end{proof}

For our Whitney decomposition we follow Lemma 2.8 in \cite{MR3183648}, whose proof is based on ideas from \cite{MR1232192}*{Lemma 2} and \cite{MR0499948}*{Theorem 1.3}. See also \cite{MR1373594}*{Lemma 5} and \cite{MR1791462}*{Lemma 1.3.3}.
\begin{lemma}\label{lem:whiba} Let $D\subset X$ be an open, nonempty, proper subset of $X.$ Then there exists a collection $\mathcal{W}(D) = \left\{B_k=B(x_k, r_k)\right\}_k$ of balls with the following properties:
\begin{enumerate}[label=\normalfont{(\roman*)}]
\item \label{coveringdisjoint} the balls $ \lbrace B(x_k, r_k/4) \rbrace_k$ are pairwise disjoint and $\bigcup_kB_k=\bigcup_k 2 B_k=D$; 
\item \label{comparabledistance}$2\rad B_k \leq \dist\left(x, X \setminus D\right)\leq 6\rad B_k$ for every $x\in 2 B_k$;
\item \label{overlappingbounded} for each $B\in \mathcal{W}(D)$, there are at most $N=N(C_d)<\infty$ balls in $\mathcal{W}(D)$ that intersect $B$.
\end{enumerate}
Furthermore, let $B_1, B_2\in \mathcal{W}(D)$ such that $B_1\cap B_2\neq\emptyset$. We have  
\begin{enumerate}[label=\normalfont{(\roman*)}]
\setcounter{enumi}{3}
\item \label{radii} $\frac{1}{4}\rad B_1 \leq \rad B_2 \leq 4 \rad B_1 $;
\item \label{measures} $\mu(B_1)\approx\mu(B_2)$.
\end{enumerate}\hfill\qed
\end{lemma}

Properties \ref{radii} and \ref{measures} are not explicitly stated in \cite{MR3183648}*{Lemma 2.8}, but they are direct consequences of \ref{comparabledistance} and Lemma \ref{lem:rmu}. The next lemma pertains to balls whose radius is comparable to their distance from the boundary. This, of course, includes but is not limited to actual Whitney balls.
\begin{lemma}\label{lem:toolsforlemma2} Let $D\subset X$ be open and proper, and $B \subset D$ a ball such that $\rad(B) \approx\dist(B, X \setminus D).$ Then
\begin{enumerate}[label=\normalfont{(\roman*)}]
\item \label{ztempone} if $B'\in\mathcal{W}(D)$ and $B' \cap B \neq \emptyset,$ then $\rad(B) \approx \rad(B')$; 
\item \label{ztemptwo} there are at most $N =N(C_d)<\infty$ Whitney balls on $D$ intersecting $B$;
\item \label{ztempfour} Assume further that $X$ is a geodesic space. If $B' \in \mathcal{W}(D)$ and $B'$ contains the center of $B$, then there exists a ball $B'' \subset B \cap B'$ such that $\rad(B'') \approx \rad(B)$ and $\mu(B'')\approx\mu(B)$.
\end{enumerate}
\end{lemma}

\begin{proof}~
\noindent (i) Let $y\in B \cap B'.$ Since $B'$ is a Whitney ball, we have that
\[
\rad(B') \leq \dist(B',\partial D) \leq \dist(y,X \setminus D) \leq  \diam(B)+ \dist(B,X \setminus D)  \approx \rad(B).
\]
Similarly, we obtain $\rad(B')\gtrsim\rad(B)$. 

\noindent (ii) Let $\mathcal{B}= \lbrace R \mathbin{:} R\in \mathcal{W}(D), \, R \cap B \neq \emptyset \rbrace$ and let $N$ denote the cardinal of $\mathcal{B}.$ Let us write $\mathcal{B}= \lbrace B_i \rbrace_{i=1}^N.$ By claim \ref{ztempone} we have that $\rad(B_i) \approx \rad(B)$ for each $i$, and by Lemma \ref{lem:rmu} $\mu(B_i) \approx \mu (B)$ for every $i$. Thus there exists a constant $\lambda \geq 1$ such that 
$$
\bigcup_{i=1}^N B_i  \subseteq \lambda B.
$$
The $B_i$ being Whitney balls, the collection $\lbrace \frac{1}{4} B_i \rbrace_i$ is pairwise disjoint. Also, observe that $\mu\left( \frac{1}{4} B_i \right) \approx  \mu (B_i) \approx \mu (B)$ for each $i$. Since we obviously have the inclusion $\bigcup_{i=1}^N \frac{1}{4} B_i \subseteq \lambda B,$ we may write
$$
\mu (B) \gtrsim \mu(\lambda B) \geq \mu \left( \bigcup_{i=1}^N \tfrac{1}{4} B_i  \right) = \sum_{i=1}^N \mu(\tfrac{1}{4} B_i) \gtrsim \sum_{i=1}^N \mu(B) = N \mu(B).
$$
This proves that $N$ is bounded above by a constant only depending on the doubling constant. 

\noindent (iii) We know from \ref{ztempone} above that $r(B) \approx r(B').$ Lemma \ref{lem:ballintersection} provides a ball $B'' \subset B \cap B'$ such that $r(B'') \approx r(B) \approx r(B').$ The statement for measures then follows from Lemma \ref{lem:rmu}.
\end{proof}

By a \emph{domain} $D$ of $X$ we understand a nonempty proper open subset of $X$ with the property that every two points in $D$ can be joined by a rectifiable curve entirely contained in $D.$

\begin{dfn} Let $D\subset X$ be a domain, $k \in \{0, 1, 2,\ldots\}$, and $B_j\in \mathcal{W}(D)$ for $j=0,\ldots,k.$ We say that 
\begin{equation*}
\mathcal{C}(B_0, B_k) = \left(B_0, \ldots, B_k\right)
\end{equation*}
is a (Whitney) \emph{chain} joining $B_0$ to $B_k$, if $B_j\cap B_{j-1}\neq\emptyset$ for every $j\in \lbrace 1, \ldots, k \rbrace.$ In this case, we say that $k$ is the \emph{length} of the chain $\mathcal{C}(B_0, B_k).$ The length of the shortest chain in $D$ from $B_0$ to $B_k$ is denoted by $\widetilde{k_D}(B_0, B_k)$. Because there is no possibility of confusion, we drop the subscript $D$ from now on. 
\end{dfn}

We will be measuring distances in $D$ in terms of the quasihyperbolic metric, which was introduced by Gehring in the 1970s to study quasiconformal mappings in $\mathbb{R}^n$; see \cite{MR437753} and \cite{MR581801}. 
\begin{dfn}Let $X$ be a geodesic space. For a domain $D\subset X$ and two points $x_1, x_2\in D$, the \emph{quasihyperbolic distance} between them is 
\begin{equation*}
k_D(x_1, x_2) = \inf_\gamma \int_\gamma \frac{\dd s}{\dist(y, \partial D)},
\end{equation*}
where the infimum is taken over all rectifiable curves $\gamma\subset D$ with endpoints $x_1$ and $x_2$. The quantity $k_D$ satisfies the axioms of a metric on $D \times D.$ A rectifiable curve $\gamma: \left[0,1\right]\to D$ is called a \emph{quasihyperbolic geodesic} if, for each pair of points $y_1, y_2\in\gamma$, it holds that 
\begin{equation*}
k_D(y_1, y_2) = \int_{\gamma|_{\left[y_1, y_2\right]}} \frac{\dd s}{\dist(y, \partial D)}.
\end{equation*}
\end{dfn}
If $E_1, E_2$ are subsets of $D$, we define $k_D(E_1, E_2) = \inf_{\substack{x_1\in E_1,\,x_2\in E_2}}k_D(x_1, x_2)$. As there is no risk of ambiguity, we will leave out the subscript $D$ in the following.

It is easy to see that the quasihyperbolic diameter of any Whitney-like ball is bounded, which is the content of the following lemma.
\begin{lemma}\label{qhdiam}
Assume further that $X$ is a geodesic space and let $D\subset X$ be a domain. If $B \subset D$ is a ball such that $d(B, \partial D) \approx \rad(B),$ then $k(x,y) \leq C$ for any two points $x,y\in B.$
\end{lemma}
\begin{proof}
Let $z$ denote the center of $B,$ and let $\gamma \subset B$ be a rectifiable curve connecting $z$ and $x$ such that $\ell(\gamma|_{[z, x]}) = d(z,x)$. Then 
\begin{equation*}
k(z, x) \leq \int_{\gamma|_{[z, x]}} \frac{\dd s}{\dist(y, \partial D)} \lesssim \int_{\gamma|_{[z, x]}} \frac{\dd s}{\rad (B)} = \frac{ \ell(\gamma|_{[z, x]})}{\rad(B)} \leq C.
\end{equation*}
Similarly we obtain $k(z, y) \leq C$, and the triangle inequality implies $k(x, y)\leq C.$
\end{proof}

The next lemma establishes an equivalence between shortest Whitney chains and quasihyperbolic distance. It is essentially contained in the proof of Lemma 9 in \cite{MR1373594}. For a detailed proof of the corresponding lemma in $\mathbb{R}^n$, see Proposition 6.1 in \cite{MR978019}. Notice that if the space $X$ is geodesic and $D\subset X$ is a proper subset, the distance functions $\dist(\cdot, \partial D)$ and $\dist(\cdot, X\setminus D)$ coincide over $D.$ We are then allowed to use Lemmas \ref{lem:whiba} and \ref{lem:toolsforlemma2} with the distance $\dist(\cdot, \partial D)$ instead of $\dist(\cdot, X\setminus D).$
\begin{lemma}\label{lem:whitneybolic} Assume further that $X$ is a geodesic space. Let $D\subset X$ be a domain and $B_i=B(x_i, r_i)\in\mathcal{W}(D),$ $i=1,2$. Then $\widetilde{k}(B_1, B_2)\approx k(x_1, x_2)$. 
\end{lemma}
\begin{proof} Let $M = \widetilde{k}(B_1, B_2)$ be the length of the shortest Whitney chain joining $B_1$ to $B_2$. In the case $x_1=x_2,$ both quantities amount to zero and there is nothing to prove. Suppose now $x_1$ and $x_2$ are distinct points. First, we prove $\widetilde{k}(B_1, B_2)\lesssim k(x_1, x_2)$. Denote by $\gamma$ the quasihyperbolic geodesic joining $x_1$ and $x_2$, and take $z$ to be an arbitrary point on $\gamma$. Of all the Whitney balls containing $z$, we choose the one with the smallest radius, say, $B=B(x,r)$. Consider the ball $B_z$ centered at $z$ and with radius $r$. It is clear that $B_z\subset 2B,$ and thus $B_z$ is contained in $D$ with $\dist(B_z,\partial D) \geq \dist(2B, \partial D) \geq r$ by virtue of Lemma \ref{lem:whiba} \ref{comparabledistance}. Also, by the properties of the Whitney decomposition (Lemma \ref{lem:whiba} \ref{comparabledistance}), we have
$$
\dist(B_z, \partial D)\leq \dist (z,\partial D) \leq \dist(B,\partial D) + \diam(B) \leq 8 r,
$$ 
and we conclude that $\dist(B_z,\partial D) \approx \rad(B_z) = r.$

Let $\gamma_z$ be a subarc of $\gamma  \cap B_z$ passing through $z$ and of maximal length. We claim that $\ell(\gamma_z)\geq C_1r$ at all times. Whenever $\gamma$ is not entirely contained in $B_z$, by the continuity of $\gamma,$ there exists a point $q\in \gamma_z$ such that $d(q,z)>r/2.$ Then we have $\ell(\gamma_z)\geq d(q,z) = r/2.$ In the case $\gamma \subset B_z$, by the properties of the Whitney decomposition there exists a constant $0<c<1$ such that $\ell(\gamma_z) = \ell(\gamma) \geq d(x_1,x_2) \geq cr_1$. Furthermore, Lemma \ref{lem:toolsforlemma2} \ref{ztempone} gives $r\approx r_1$ and consequently $\ell(\gamma_z) \geq C_1r$. Recalling that $\gamma_z\subset B_z$ and $\dist (z,\partial D) \leq 8r$, in all cases it holds that
\begin{equation}\label{tempkayra}
\int_{\gamma_z}\frac{\dd l}{\dist(y,\partial D)} \geq \frac{\ell(\gamma_z)}{r+\dist( z,\partial D)}\geq \frac{C_1r}{9r} \geq C_2.
\end{equation}
Next, we cover the geodesic $\gamma$ by balls $\lbrace B_{z_i} \rbrace_i$, with the points $\{z_i\}_i\subset \gamma$ chosen so that every point is contained in at most two balls $B_{z_i}$. Among these collections we choose the one with the smallest cardinality, say $m=\card\lbrace B_{z_i}\rbrace$. For any $z\in \gamma$, Lemma \ref{lem:toolsforlemma2} \ref{ztemptwo} shows that there are at most $C$ Whitney balls intersecting $B_z$. Now let $M_1$ be the minimal number of Whitney balls needed to cover $\bigcup_i{B_{z_i}}$, and denote this collection by $\mathcal{F}$. Clearly $M_1\geq M$, because $M$ was the length of the shortest chain joining $B_1$ and $B_2.$ Also, we have that $\card\mathcal{F} = M_1$ and, by minimality, for every $B\in \mathcal{F}$ there is at least one $i$ such that $B \cap B_{z_i} \neq \emptyset$. Therefore, we have that $\mathcal{F} \subset \bigcup_i \lbrace B \in \mathcal{W}(D) \mathbin{:} B \cap B_{z_i} \neq \emptyset \rbrace$ and
$$
M_1 = \card\mathcal{F} \leq \card \left( \bigcup_{i=1}^m \lbrace B \in \mathcal{W}_D \mathbin{:} B \cap B_{z_i} \neq \emptyset \rbrace \right) \leq \sum_{i=1}^m \card \lbrace  B \in \mathcal{W}_D \mathbin{:} B \cap B_{z_i} \neq \emptyset \rbrace \leq Cm.
$$
We obtain $Cm \geq M_1 \geq M.$ Now, denoting $\gamma_i = \gamma\cap B_{z_i}$ and applying \eqref{tempkayra} on each of these subarcs, we obtain the estimate
$$
 k(x_1,x_2) = \int_\gamma \frac{\dd l}{\dist(y,\partial D)} \geq \frac{1}{2} \sum_{i=1}^m \int_{\gamma_i} \frac{\dd l}{\dist(y,\partial D)} \geq \frac{mC_2}{2}\geq \frac{C_2}{2C} M  = C_3\widetilde{k}(B_1, B_2). 
$$

As for the inequality in the other direction, take the the shortest chain $\mathcal{C}=\left( B^1, B^2, \ldots, B^M \right)$ connecting $B_1 =B^1= B(x^1, r^1)$ and $B_2 =B^M= B(x^M, r^M).$ For every $j \in \lbrace 1,\ldots, M-1 \rbrace,$ take a point $p_j$ in $B^j \cap B^{j+1}.$ We have $k(x^j,p_j)\leq C$ and $k(x^{j+1},p_j) \leq C$ owing to Lemma \ref{qhdiam}. Using the triangle inequality repeatedly, we obtain
\[
k(x_1, x_2) = k(x^1, x^M) \leq \sum_{j=1}^{M-1} \left( k(x^j,p_j) + k(p_j,x^{j+1}) \right) \leq 2C(M-1) \lesssim M= \widetilde{k}(B_1, B_2),
\]
whereby the statement is proven.
\end{proof}

\section{Estimates for weights on Whitney chains}\label{section:weights}

In a metric measure space $X$ we call a domain $D\subset X$ an \emph{extension domain} for the Muckenhoupt class $A_p$, if whenever $w\in A_p(D)$ there exists a $W\in A_p(X)$ such that $W=w$ a.~e. on $D$. Holden \cite{MR1162041} gives certain sufficient conditions for extension domains in $\mathbb{R}^n$. Holden's strategy of proof is to verify Wolff's condition \eqref{intro:muck} by propagating estimates on cubes along Whitney chains. In this final section we adapt \cite{MR1162041}*{Lemma 2} into the metric setting, resulting in Lemma \ref{lem:hold2} below. In Holden's Euclidean argument, \cite{MR1162041}*{Lemma 2} is used to estimate integrals over each cube in a dyadic decomposition of $Q\cap E$ in terms of integrals over cubes arising from Holden's assumptions that enjoy additional good properties.

To put the extension problem in context, it is instructive to outline the situation regarding the space of \emph{functions of bounded mean oscillation} (BMO). These are intimately related to Muckenhoupt weights: whenever a weight $w$ belongs to $A_p$, then $\log w$ is of bounded mean oscillation. Conversely, whenever $f\in \mathrm{BMO}$, then $\exp(\delta f)\in A_p$ for small enough $\delta$. Peter W. Jones \cite{MR554817} has shown that extension domains for BMO functions in the Euclidean space $\mathbb{R}^n$ are precisely \emph{uniform domains}, that can be characterized in terms of the quasihyperbolic metric. Vodop'yanov and Greshnov \cite{MR1373594} extended Jones' characterization to metric spaces supporting a doubling measure. Recently, Butaev and Dafni \cite{buda} proved the analogue of Jones' characterization for functions of \emph{vanishing mean oscillation} in $\mathbb{R}^n$. 

For Muckenhoupt weights, the question remains open. Some examples and counterexamples concerning necessary or sufficient conditions for extension domains for $A_p$ are discussed by Holden \cite{MR1162041} and Koskela in his corresponding review \cite{koskela}.

For the purposes of this section we need to introduce classical $A_p$ weights defined on a subset. Compare this to Definition \ref{def:Aptilde}. By $L^1_{\loc}(D)$ we denote the class of functions that are integrable on every compact subset of $D.$
\begin{dfn}\label{def:Ap} Let $D\subset X$ be a nonempty open subset in a metric space $X$, and $1<p<\infty$. An a.~e. positive function $w\in L^1_{\loc}(D)$ is called a \emph{Muckenhoupt $A_p$ weight} in $D$, denoted $w\in A_p(D)$, if
\begin{equation}\label{Ap}
\normp{w}{p} = 
\sup_{\substack{B\subset D}}\left(\frac{1}{\mu(B)}\int_{B}w \dd\mu \right)\left(\frac{1}{\mu(B)}\int_{B}w ^{-\frac{1}{p-1}}\dd \mu \right)^{p-1}<\infty.
\end{equation}
The supremum is taken over all balls $B\subset D$. For $p=1$, a nonnegative function $w\in L^1_{\loc}(D)$ belongs to $A_1(D)$ if there exists a constant $C>0$ such that for all balls $B\subset D$
\begin{equation}\label{A1}
\frac{1}{\mu(B)}\int_{B}w \dd\mu  \leq C\essinf_{ B} w.
\end{equation}
We denote by $\normp{w}{1}$ the infimum of the $C>0$ for which the inequality \eqref{A1} holds. 
\end{dfn}

Provided that the underlying measure $\mu$ satisfies a doubling condition and $w$ is an $A_p$ weight, the weighted measure $w\dd\mu$ satisfies the doubling condition for balls $B$ such that $2B \subset D.$ This property follows from statement \ref{five} of the next lemma, that collects some estimates for weights on balls and chains. Throughout the rest of the section, we will assume that $(X, d, \mu)$ is a complete metric measure space such that $\mu$ satisfies \eqref{Doubling}.

\begin{lemma}\label{lem:hold1} Let $D\subset X$ be an open proper subset, and $w\in A_p(D)$ with $1\leq p<\infty$.
\begin{enumerate}[label=\normalfont{(\roman*)}]
\item \label{one} If the ball $B\subset D$, then
\begin{equation*}
\frac{1}{\mu(B)}\int_B w \dd \mu \leq \normp{w}{p} \exp\left(\frac{1}{\mu(B)}\int_B\log w \dd\mu \right).
\end{equation*}
\item \label{five} If $B$ is a ball in $D$ and $E \subset B$ is a measurable subset with $\mu(E)>0,$ then
\begin{equation*}
\int_{B}w \dd\mu  \leq \normp{w}{p}\left(\frac{\mu(B)}{\mu(E)}\right)^{p} \int_{E}w \dd\mu . 
\end{equation*}
\item\label{four} \emph{(the $A_\infty$ condition)} There exist constants $0< C_w,\,\delta(w)<\infty$, depending only on the doubling constant $C_d(\mu)$ and the weight $w$, such that for all balls $B\subset D$ and all measurable subsets $E \subset B$ we have
\begin{equation*}
\frac{w(E)}{w(B)}\leq C_w\left(\frac{\mu(E)}{\mu(B)}\right)^{\delta(w)}. 
\end{equation*}
\item \label{three} Assume further that $D$ is a domain. If $B_1$, $B_2\in \mathcal{W}(D)$, then 
\begin{equation*}%
\frac{1}{\mu(B_1)}\int_{B_1}w \dd\mu \leq \exp\left( C \widetilde{k}(B_1, B_2)\right)\frac{1}{\mu(B_2)}\int_{B_2}w \dd\mu,
\end{equation*}
where $C$ is a constant only depending on $C_d,$ $p,$ and $\normp{w}{p}.$
\end{enumerate}
\end{lemma}
\begin{proof}
To prove \ref{one}, we may assume that $p>1$ because $A_1(D) \subset A_2(D).$ Now, the inequality \ref{one} follows from the $A_p$ condition (Definition \ref{def:Ap}). Indeed, notice that the function $t\mapsto \exp\left(-t(p-1)^{-1}\right)$ is convex, and apply Jensen's inequality:
\begin{align*}
\normp{w}{p} &\geq \left(\frac{1}{\mu(B)}\int_Bw \dd\mu \right)\left(\frac{1}{\mu(B)}\int_B\exp\left(\log w \right)^{-\frac{1}{p-1}}\dd\mu \right)^{p-1}\\
& \geq \left(\frac{1}{\mu(B)}\int_B w \dd\mu \right)\exp\left(\frac{1}{\mu(B)}\int_B\log w \dd\mu \right)^{-\frac{1}{p-1}(p-1)}. 
\end{align*}
When $p>1$, the statement \ref{five} is a consequence of the $A_p$ condition (Definition \ref{def:Ap}):
\begin{align*}
& \frac{1}{\mu(B)}\int_{B}w\dd\mu \leq \normp{w}{p}\left(\frac{1}{\mu(B)}\int_{B_1}w^{-\frac{1}{p-1}}\dd\mu\right)^{-(p-1)}\\
&\hspace*{2em} \leq \normp{w}{p}\left(\frac{\mu(E)}{\mu(B)}\right)^{-(p-1)}\left(\frac{1}{\mu(E)}\int_{E}w^{-\frac{1}{p-1}}\dd\mu\right)^{-(p-1)}\\
&\hspace*{2em} \leq \normp{w}{p}\left(\frac{\mu(B)}{\mu(E)}\right)^{p-1}\frac{1}{\mu(E)}\int_{E}w\dd\mu,
\end{align*} 
where the last estimate follows from H\"older's inequality. Besides, when $p=1,$ $w\in A_1(D)$ implies
$$
\frac{1}{\mu(B)}\int_B w \dd\mu \leq \normp{w}{1} \essinf_{B} w \leq \normp{w}{1} \essinf_{E} w \leq \frac{\normp{w}{1}}{\mu(E)} \int_E w\dd\mu.
$$

For a proof of the $A_\infty$ condition \ref{four} we refer to \cite{MR1011673}, Theorem I.15. There, the weights are globally defined in $X,$ but the proof for weights in $A_p(D)$ is exactly the same. Indeed, provided that $w$ satisfies a reverse Hölder inequality with exponent $1+\delta$ over balls $B \subset D$ (compare Proposition \ref{thm:revho}), using it and the classical Hölder inequality we have
$$
\int_E w \dd \mu \leq \mu(E)^{\frac{\delta}{1+\delta}}\left( \int_B w^{1+\delta} \dd \mu \right)^{\frac{1}{1+\delta}} \leq C \left( \frac{\mu(E)}{\mu(B)} \right)^{\frac{\delta}{1+\delta}} \int_B w \dd \mu
$$
for all balls $B \subset D$ and all measurable subsets $E \subset B.$ The fact that $w$ satisfies a reverse Hölder inequality can be proven using a version of Gehring's lemma for weights in $A_p(D),$ whose proof is similar to that for $A_p(X)$ weights because in Definition \ref{def:Ap} we only consider balls that are entirely contained in $D.$ A proof of Gehring's lemma for $A_p(X)$ weights can be found in \cite{MR2867756}, p.~77.

Finally, let us prove \ref{three}. Let $B_j = B(p_j, r_j)$ and $B_{j+1} = B(p_{j+1}, r_{j+1})$ be two consecutive balls in the chain $\mathcal{C}(B_1, B_2)$. Then $B_j\cap B_{j+1}\neq\emptyset$. To begin with, we show that there is a constant $C$ such that 
\begin{equation}\label{jfjs}
\int_{B_{j+1}} w \dd\mu \leq C\int_{B_j} w\dd\mu. 
\end{equation}
To this effect, let $y\in B_j \cap B_{j+1}$ and suppose  first that $\dist(y,p_{j+1})< \frac{1}{8} r_{j+1}$. For any $z\in X,$ we have
$$
\dist(z,p_j)  \leq  \dist(z, p_{j+1}) + \dist(p_j,y) + \dist(y, p_{j+1}) < \dist(z, p_{j+1}) + r_j+ \tfrac{1}{8} r_{j+1}.
$$
Letting $z\in \frac{1}{8}B_{j+1}=B\left(p_{j+1},\frac{1}{8}r_{j+1}\right)$ in the above and applying \ref{radii} of Lemma \ref{lem:whiba}, we have
\begin{align*}
\dist(z, p_j) <  \tfrac{1}{8} r_{j+1} + r_j+ \tfrac{1}{8} r_{j+1}  \leq \tfrac{1}{4}\cdot 4r_j+r_j = 2r_j,
\end{align*}
which shows that $\frac{1}{8}B_{j+1} \subset 2B_j \subset D$ as guaranteed by \ref{coveringdisjoint} of Lemma \ref{lem:whiba}. Using \ref{five} of the current lemma and the fact that $\mu$ is doubling, we may write
\begin{align*}
\nonumber & \int_{B_{j+1}} w\dd\mu \lesssim \normp{w}{p} \left( \frac{\mu(B_{j+1})}{\mu(\frac{1}{8}B_{j+1})} \right)^p \int_{\frac{1}{8}B_{j+1}} w \dd\mu
\lesssim \normp{w}{p} \int_{\frac{1}{8}B_{j+1}} w \dd\mu \\
& \hspace*{2em} \leq \normp{w}{p} \int_{2B_j} w\dd\mu
\lesssim \normp{w}{p} \left(\frac{\mu(2B_j)}{\mu(B_j)} \right)^p \int_{B_j} w\dd\mu
\lesssim \normp{w}{p} \int_{B_j} w\dd\mu,
\end{align*}
which proves \eqref{jfjs}.

Now suppose that $\dist(y,p_{j+1}) \geq \frac{1}{8} r(B_{j+1}).$ The balls $B_{j+1}^*=B(y, \frac{1}{8}r_{j+1})$ and $\frac{1}{8} B_{j+1}$ have the same radius and $\dist(y,p_{j+1}) \approx\frac{1}{8} r_{j+1}.$ Then, by Lemma \ref{lem:rmu} and the doubling condition \eqref{Doubling}, it holds that
$$
\mu(B_{j+1}^*)\approx \mu  \left( \tfrac{1}{8} B_{j+1} \right) \approx \mu (B_{j+1}).
$$
Using the triangle inequality, Lemma \ref{lem:whiba} \ref{radii}, and the fact that $y\in B_j \cap B_{j+1},$ we easily obtain for a $z\in B_{j+1}^*$
\begin{align*}
\dist(z,p_{j+1}) & \leq \dist(z,y)+\dist(y, p_{j+1}) \leq\tfrac{1}{8}r_{j+1}+r_{j+1} < 2r_{j+1},\\ 
\dist(z,p_j) & \leq \dist(z,y)+\dist(y, p_j) \leq\tfrac{1}{8}r_{j+1}+r_j \leq \tfrac{3}{2}r_j < 2r_j,\\ 
\end{align*}
which implies that $B_{j+1}^* \subset 2B_{j+1} \cap 2B_j.$ Applying \ref{five} of the current lemma, we conclude that
\begin{align*}
& \int_{B_{j+1}} w\dd\mu \leq \int_{2B_{j+1}} w\dd\mu \lesssim \normp{w}{p} \left( \frac{\mu(2B_{j+1})}{\mu(B_{j+1}^*)} \right)^p \int_{B_{j+1}^*} w\dd\mu \lesssim \normp{w}{p} \int_{B_{j+1}^*} w\dd\mu \\
& \hspace*{2em} \leq \normp{w}{p} \int_{2B_j } w\dd\mu \lesssim \normp{w}{p} \left(\frac{2B_j)}{\mu(B_j)} \right)^p \int_{B_j} w\dd\mu\lesssim \normp{w}{p} \int_{B_j} w \dd\mu.
\end{align*}
Thus, in any case we have $\int_{B_{j+1}} w\dd\mu \lesssim \normp{w}{p} \int_{B_j} w\dd\mu.$ Reversing the roles of $B_j$ and $B_{j+1}$, we obtain $\int_{B_j} w\dd\mu \approx \int_{B_{j+1}} w\dd\mu$, where the constants involved depend on $p$, the doubling constant, and $\normp{w}{p}.$ Furthermore, by Lemma \ref{lem:whiba} \ref{measures}, there exists a constant $C_1$ such that  
\begin{equation*}
\frac{1}{\mu(B_j)}\int_{B_j} w\dd\mu \leq C_1 \frac{1}{\mu(B_{j+1})} \int_{B_{j+1}} w\dd\mu. 
\end{equation*}
Recalling that $\widetilde{k}(B_1, B_2)$ is the number of balls in $\mathcal{W}(D)$ in the shortest chain from $B_1$ to $B_2$, we apply this recursively to obtain 
\begin{equation*}
\frac{1}{\mu(B_1)}\int_{B_1} w\dd\mu \leq C_1^{\widetilde{k}(B_1, B_2)} \frac{1}{\mu(B_2)} \int_{B_2} w\dd\mu. 
\end{equation*}
Choose $C = \log C_1$ to get the desired expression.
\end{proof}
\begin{remark}In the proof of property \ref{three} above, we in fact showed that whenever $\mu$ is doubling and $w\in A_p(D)$, then for any two Whitney balls $B_1, B_2\in \mathcal{W}(D)$ such that $B_1\cap B_2\neq\emptyset$ it holds that $\int_{B_1}w \approx\int_{B_2}w.$
Statement \ref{five} is a ``reverse $A_\infty$ condition'' that follows immediately from the $A_p$ condition. Namely, whenever $E \subset B$, then
\[\frac{\mu(E)}{\mu(B)}\leq C(w) \left(\frac{w(E)}{w(B)}\right)^\frac{1}{p}.\] 
\end{remark}

\begin{lemma}\label{lem:hold2}
Assume further that $X$ is a geodesic space and let $D\subset X$ be a domain, $w\in A_p(D)$ with $1\leq p<\infty$, $C$ a constant possibly depending on $C_d$, and $B_1, B_2\subset D$ balls satisfying
\begin{enumerate}[label=\normalfont{(\roman*)}]
\item \label{lem:hold2.ass1} $\dist(B_i, \partial D)\approx \rad(B_i)$, $i=1,2$,
\item \label{lem:hold2.ass2} $k(B_1, B_2)\leq C$.
\end{enumerate}
Then 
\begin{equation*}
\int_{B_1}w \dd\mu \approx \int_{B_2}w \dd\mu,
\end{equation*}
where the constants involved depend on $C,$ $C_d,$ $p,$ and $\normp{w}{p}.$
\end{lemma}
\begin{proof}
Let $B_i'=B(z_i', r_i')\in\mathcal{W}(D)$, $i=1,2$, contain the centers of $B_1=B(z_1,r_1)$ and $B_2=B(z_2, r_2)$ respectively. Lemma \ref{lem:toolsforlemma2} \ref{ztempfour} guarantees that there exist balls $B_i''\subset B_i\cap B_i'$, $i=1,2$, such that $\rad(B_i'')\approx\rad(B_i')\approx\rad(B_i)$ and $\mu(B_i'')\approx\mu(B_i')\approx\mu(B_i)$. 

Furthermore, it holds that $k(z_1', z_2')\leq C$. To see this, let $x_1, x_2$ be points contained in $B_1$ and $B_2$ respectively such that $k(x_1,x_2) \leq  k(B_1, B_2)+C.$ Using the triangle inequality for $k$ we have
\begin{equation*}
k(z_1', z_2') \leq k(z_1', z_1) + k(z_1, x_1) + k(x_1, x_2) + k(x_2, z_2) + k(z_2, z_2').
\end{equation*}
Observe that $z_i\in B_i'$, $x_i\in B_i$, and $k(x_1, x_2)\leq 2C$. By Lemma \ref{qhdiam}, the quasihyperbolic diameters of the balls $B_i$ and $B_i'$ are uniformly bounded, and we have
\[k(z_1', z_2') \leq C_1.\]
Also, by Lemma \ref{lem:whitneybolic}, we have $k(z_1', z_2')\approx \widetilde{k}(B_1',B_2')$ and thus $\widetilde{k}(B_1',B_2')\lesssim C_1.$ With these remarks, Lemma \ref{lem:hold1} \ref{five} allows us to estimate
\begin{align}
\label{yydbl}& \frac{1}{\mu(B_1)}\int_{B_1}w \dd\mu \lesssim \left(\frac{\mu(B_1)}{\mu(B_1'')}\right)^{p-1} \frac{1}{\mu(B_1'')}\int_{B_1''}w \dd\mu \\
\nonumber \hspace*{2em} & \lesssim \frac{\mu(B_1')}{\mu(B_1'')}\frac{1}{\mu(B_1')}\int_{B_1'}w \dd\mu\\
\label{same} \hspace*{2em} & \lesssim \frac{1}{\mu(B_1')}\int_{B_1'}w \dd\mu \\
\label{yylem}\hspace*{2em}  & \lesssim \frac{1}{\mu(B_2')}\int_{B_2'}w \dd\mu \\
\label{yylemm} \hspace*{2em} & \lesssim \left(\frac{\mu(B_2')}{\mu(B_2'')}\right)^{p-1}\frac{1}{\mu(B_2'')}\int_{B_2''}w \dd\mu \\
\label{samee} \hspace*{2em} & \lesssim \frac{1}{\mu(B_2'')}\int_{B_2''}w \dd\mu \\
\nonumber \hspace*{2em} & \lesssim \frac{1}{\mu(B_2)}\int_{B_2}w \dd\mu .
\end{align}
Line \eqref{yydbl} follows from the fact that the measure $w\dd\mu$ is doubling, while \eqref{yylem} and \eqref{yylemm} are Lemma \ref{lem:hold1} \ref{three} and \ref{five}, respectively. On lines \eqref{same} and \eqref{samee} we used the fact that $\mu(B_i'')\approx\mu(B_i')\approx\mu(B_i)$. 

Finally, if $\left(B_1'= B^0,\ldots, B^N =B_2'\right)$ is the shortest Whitney chain connecting $B_1'$ and $B_2',$ we have that $N \lesssim C$ by the previous arguments. Since each pair of consecutive balls $(B^{j-1},$ $B^j)$ in the chain has nonempty intersection, we have $\rad(B^{j-1}) \approx \rad(B^j)$ by Lemma \ref{lem:whiba}\ref{radii} and therefore $\rad(B^0) \approx \rad(B^j) \approx \rad(B^N)$ for every for every $j=1,\ldots, N,$ because $N \lesssim C.$ Moreover, if $p_j \in B^{j-1} \cap B^j,$ the triangle inequality gives 
$$
\dist(p_0, p_N) \leq \sum_{j=1}^N \dist(p_{j-1}, p_j)\leq \sum_{j=1}^N 2 \rad(B_{j-1}) \lesssim \rad(B^N).
$$
It follows from Lemma \ref{lem:rmu} that $\mu(B_1')\approx \mu(B_2'),$ which in turn implies $\mu(B_1)\approx\mu(B_2).$ We conclude that 
$$
\int_{B_1}w \dd\mu  \lesssim \int_{B_2}w \dd\mu 
$$
and, swapping the roles of $B_1$ and $B_2$, the inequality in the other direction.
\end{proof}

\end{document}